\documentclass{kms-j}

\usepackage{psfrag,graphics,graphicx}

\theoremstyle{plain}

\theoremstyle{definition}

\theoremstyle{remark}

\begin{document}

\title[The competition numbers of Hamming graphs]
{The competition numbers of Hamming graphs
with diameter at most three}

\author{Boram PARK}
\address{
Department of Mathematics Education\\
Seoul National University\\
Seoul 151-742, Korea}
\email{kawa22@snu.ac.kr}
\author{Yoshio SANO}
\address{
Pohang Mathematics Institute \\
Pohang University of Science and Technology \\
Pohang 790-784, Korea} 
\email{ysano@postech.ac.kr}
\thanks{
This work was supported by Priority Research Centers Program
through the National Research Foundation of Korea (NRF) funded by 
the Ministry of Education, Science and Technology (2009-0094069) 
and 
by the Korea Research Foundation Grant
funded by the Korean Government (MOEHRD)
(KRF-2008-531-C00004).
The first author was supported by Seoul Fellowship.}


\begin{abstract}

The competition graph of a digraph $D$ is a graph which
has the same vertex set as $D$
and has an edge between $x$ and $y$ if and only if
there exists a vertex $v$ in $D$ such that
$(x,v)$ and $(y,v)$ are arcs of $D$.
For any graph $G$, $G$ together with sufficiently many isolated vertices
is the competition graph of some acyclic digraph.
The competition number $k(G)$ of a graph $G$
is defined to be the smallest number of such isolated vertices.
In general, it is hard to compute the competition number
$k(G)$ for a graph $G$ and it has been one of
important research problems in the study of
competition graphs.
In this paper, we compute the competition numbers
of Hamming graphs with diameter at most three.
\end{abstract}

\maketitle

\newtheorem{Thm}{Theorem}
\newtheorem{Lem}[Thm]{Lemma}
\newtheorem{Prop}[Thm]{Proposition}
\newtheorem{Cor}[Thm]{Corollary}
\newtheorem{Conj}[Thm]{Conjecture}
\newtheorem{Rem}[Thm]{Remark}

\newtheorem{Clm}{Claim}
\newtheorem*{Defi}{Definition}

\newcommand{\bA}{\ensuremath{\mathbb{A}}}
\newcommand{\bB}{\ensuremath{\mathbb{B}}}
\newcommand{\bC}{\ensuremath{\mathbb{C}}}
\newcommand{\bD}{\ensuremath{\mathbb{D}}}
\newcommand{\bE}{\ensuremath{\mathbb{E}}}
\newcommand{\bF}{\ensuremath{\mathbb{F}}}
\newcommand{\bG}{\ensuremath{\mathbb{G}}}
\newcommand{\bH}{\ensuremath{\mathbb{H}}}
\newcommand{\bI}{\ensuremath{\mathbb{I}}}
\newcommand{\bJ}{\ensuremath{\mathbb{J}}}
\newcommand{\bK}{\ensuremath{\mathbb{K}}}
\newcommand{\bL}{\ensuremath{\mathbb{L}}}
\newcommand{\bM}{\ensuremath{\mathbb{M}}}
\newcommand{\bN}{\ensuremath{\mathbb{N}}}
\newcommand{\bO}{\ensuremath{\mathbb{O}}}
\newcommand{\bP}{\ensuremath{\mathbb{P}}}
\newcommand{\bQ}{\ensuremath{\mathbb{Q}}}
\newcommand{\bR}{\ensuremath{\mathbb{R}}}
\newcommand{\bS}{\ensuremath{\mathbb{S}}}
\newcommand{\bT}{\ensuremath{\mathbb{T}}}
\newcommand{\bU}{\ensuremath{\mathbb{U}}}
\newcommand{\bV}{\ensuremath{\mathbb{V}}}
\newcommand{\bW}{\ensuremath{\mathbb{W}}}
\newcommand{\bX}{\ensuremath{\mathbb{X}}}
\newcommand{\bY}{\ensuremath{\mathbb{Y}}}
\newcommand{\bZ}{\ensuremath{\mathbb{Z}}}

\newcommand{\cA}{\ensuremath{\mathcal{A}}}
\newcommand{\cB}{\ensuremath{\mathcal{B}}}
\newcommand{\cC}{\ensuremath{\mathcal{C}}}
\newcommand{\cD}{\ensuremath{\mathcal{D}}}
\newcommand{\cE}{\ensuremath{\mathcal{E}}}
\newcommand{\cF}{\ensuremath{\mathcal{F}}}
\newcommand{\cG}{\ensuremath{\mathcal{G}}}
\newcommand{\cH}{\ensuremath{\mathcal{H}}}
\newcommand{\cI}{\ensuremath{\mathcal{I}}}
\newcommand{\cJ}{\ensuremath{\mathcal{J}}}
\newcommand{\cK}{\ensuremath{\mathcal{K}}}
\newcommand{\cL}{\ensuremath{\mathcal{L}}}
\newcommand{\cM}{\ensuremath{\mathcal{M}}}
\newcommand{\cN}{\ensuremath{\mathcal{N}}}
\newcommand{\cO}{\ensuremath{\mathcal{O}}}
\newcommand{\cP}{\ensuremath{\mathcal{P}}}
\newcommand{\cQ}{\ensuremath{\mathcal{Q}}}
\newcommand{\cR}{\ensuremath{\mathcal{R}}}
\newcommand{\cS}{\ensuremath{\mathcal{S}}}
\newcommand{\cT}{\ensuremath{\mathcal{T}}}
\newcommand{\cU}{\ensuremath{\mathcal{U}}}
\newcommand{\cV}{\ensuremath{\mathcal{V}}}
\newcommand{\cW}{\ensuremath{\mathcal{W}}}
\newcommand{\cX}{\ensuremath{\mathcal{X}}}
\newcommand{\cY}{\ensuremath{\mathcal{Y}}}
\newcommand{\cZ}{\ensuremath{\mathcal{Z}}}

\newcommand{\gA}{\ensuremath{\mathfrak{A}}}
\newcommand{\gB}{\ensuremath{\mathfrak{B}}}
\newcommand{\gC}{\ensuremath{\mathfrak{C}}}
\newcommand{\gD}{\ensuremath{\mathfrak{D}}}
\newcommand{\gE}{\ensuremath{\mathfrak{E}}}
\newcommand{\gF}{\ensuremath{\mathfrak{F}}}
\newcommand{\gG}{\ensuremath{\mathfrak{G}}}
\newcommand{\gH}{\ensuremath{\mathfrak{H}}}
\newcommand{\gI}{\ensuremath{\mathfrak{I}}}
\newcommand{\gJ}{\ensuremath{\mathfrak{J}}}
\newcommand{\gK}{\ensuremath{\mathfrak{K}}}
\newcommand{\gL}{\ensuremath{\mathfrak{L}}}
\newcommand{\gM}{\ensuremath{\mathfrak{M}}}
\newcommand{\gN}{\ensuremath{\mathfrak{N}}}
\newcommand{\gO}{\ensuremath{\mathfrak{O}}}
\newcommand{\gP}{\ensuremath{\mathfrak{P}}}
\newcommand{\gQ}{\ensuremath{\mathfrak{Q}}}
\newcommand{\gR}{\ensuremath{\mathfrak{R}}}
\newcommand{\gS}{\ensuremath{\mathfrak{S}}}
\newcommand{\gT}{\ensuremath{\mathfrak{T}}}
\newcommand{\gU}{\ensuremath{\mathfrak{U}}}
\newcommand{\gV}{\ensuremath{\mathfrak{V}}}
\newcommand{\gW}{\ensuremath{\mathfrak{W}}}
\newcommand{\gX}{\ensuremath{\mathfrak{X}}}
\newcommand{\gY}{\ensuremath{\mathfrak{Y}}}
\newcommand{\gZ}{\ensuremath{\mathfrak{Z}}}

\section{Introduction and Main Results}

The notion of a competition graph was introduced by Cohen \cite{Cohen1}
in connection with a problem in ecology
(see also \cite{Cohen2}).
The {\it competition graph} $C(D)$ of a digraph $D$ is
the (simple undirected) graph
which has the same vertex set
as $D$ and has an edge between vertices $u$ and $v$ if and only if
there is a vertex $x$ in $D$ such that $(u,x)$ and $(v,x)$ are arcs of $D$.
For any graph $G$, $G$ together
with sufficiently many isolated vertices is the competition graph of an
acyclic digraph.
Roberts \cite{MR0504054} defined the {\em competition number} $k(G)$ of
a graph $G$ to be the smallest number $k$ such that $G$ together with
$k$ isolated vertices is the competition graph of an acyclic digraph.
Opsut \cite{MR679638} showed that the computation of the
competition number of a graph is an NP-hard problem.
In the study of competition graphs,
it has been one of important research problems to
compute the competition numbers for various graph classes
(see \cite{MR2181040}, \cite{kimsu}, \cite{MR1833332}, \cite{MR2176262}, 
\cite{LKKS}, \cite{kr}, \cite{twoholes}, \cite{MR0504054}, \cite{WL} 
for graphs whose competition numbers are known). 
For some special graph families, we have explicit formulae
for computing competition numbers.
For example, if $G$ is a chordal graph without isolated vertices 
then $k(G)=1$,
and if $G$ is a nontrivial triangle-free connected graph 
then $k(G)=|E(G)|-|V(G)|+2$ (see \cite{MR0504054}).

Recently, it has been concerned to find the competition numbers of 
interesting graph families used in many areas 
of mathematics and computer science 
(see \cite{KPS}, \cite{KimSano}, \cite{PKS1}, \cite{PKS2}, \cite{regpoly}). 
Hamming graphs are known as an interesting graph family 
in connection with error-correcting codes, association schemes, 
and several branches of mathematics. 
For a positive integer $q$,
we denote the $q$-set $\{1,2 \ldots, q \}$ by $[q]$.
Also we denote the set of $n$-tuples over $[q]$ by $[q]^n$.
For positive integers $n$ and $q$,
the {\it Hamming graph} $H(n,q)$
is the graph which has the vertex set ${[q]^n}$ and
in which
two vertices $x=(x_1,x_2,\ldots,x_n)$ and $y=(y_1,y_2,\ldots,y_n)$ 
are adjacent if $d_H(x,y)=1$,
where $d_H:[q]^n \times [q]^n \to \bZ$ is
the {\it Hamming distance} defined by
\[
d_H(x,y):=|\{i \in [n] \mid x_i \neq y_i \}|.
\]
Note that the diameter of the Hamming graph
$H(n,q)$ is equal to $n$ if $q \geq 2$.
Since the Hamming graph $H(n,q)$ is an $n(q-1)$-regular graph 
with $q^n$ vertices,
it follows that the number of edges of
the Hamming graph $H(n,q)$ is equal to $\frac{1}{2} n(q-1)q^n$.

In this paper, we study the competition numbers of Hamming graphs.
If $q=1$, then $H(n,1)$ is $K_1$ and so the following holds:

\begin{Prop}
For $n \geq 1$, we have $k(H(n,1))=0$.
\end{Prop}

\noindent
If $q=2$, then
since $H(n,2)$ triangle-free and connected,
we have
\begin{eqnarray*}
k(H(n,2)) &=& |E(H(n,2))|-|V(H(n,2))|+2 \\
&=& n 2^{n-1} -2^n +2 \\
&=& (n-2)2^{n-1} +2.
\end{eqnarray*}

\begin{Prop}\label{prop:kHn2}
For $n \geq 1$, we have $k(H(n,2))=(n-2)2^{n-1} +2$.
\end{Prop}

\noindent
If $n=1$, then $H(1,q)$ is
the complete graph $K_q$ with $q$ vertices
and so the following holds:

\begin{Prop}
For $q \geq 2$, we have $k(H(1,q))=1$.
\end{Prop}

\noindent
However, in general, it is not easy to compute $k(H(n,q))$.
In this paper, we give the exact values of $k(H(2,q))$ and $k(H(3,q))$.
Our main results are the following:

\begin{Thm} \label{thm:H(2,q)}
For $q \geq 2$, we have $k(H(2,q)) =2$.
\end{Thm}

\begin{Thm}\label{thm:H(3,q)}
For $q \geq 3$, we have $k(H(3,q)) =6$.
\end{Thm}

We use the following notation and terminology in this paper.
For a digraph $D$,
a sequence $v_1, v_2, \ldots, v_n$ of the vertices of $D$
is called an \emph{acyclic ordering} of $D$
if $(v_i,v_j) \in A(D)$ implies $i>j$.
It is well-known that a digraph $D$ is acyclic
if and only if there exists an acyclic ordering of $D$.
For a digraph $D$ and a vertex $v$ of $D$,
we define the {\it out-neighborhood}
$N^{+}_D(v)$ of $v$ in $D$ to be the set
$\{w \in V(D) \mid (v,w) \in A(D)\}$, and
the {\it in-neighborhood} $N^{-}_D(v)$
of $v$ in $D$ to be the set $\{w \in V(D) \mid (w,v)\in A(D)\}$.
A vertex in the out-neighborhood $N^{+}_D(v)$ of a vertex $v$
in a digraph $D$ is called a {\it prey} of $v$ in $D$.
For a graph $G$ and a vertex $v$ of $G$,
we define the {\it open neighborhood} $N_G(v)$ of $v$ in $G$
to be the set $\{u \in V(G) \mid uv \in E(G) \}$,
and the {\it closed neighborhood} $N_G[v]$ of $v$ in $G$
to be the set $N_G(v) \cup \{v\}$.
We denote the subgraph of $G$ induced by
$N_G(v)$ (resp. $N_G[v]$) by the same symbol $N_G(v)$ (resp. $N_G[v]$).

For a clique $S$ of a graph $G$
and an edge $e$ of $G$,
we say {\it $e$ is covered by $S$}
if both of the endpoints of $e$ are contained in $S$.
An \textit{edge clique cover} of a graph $G$
is a family of cliques of $G$ such that
each edge of $G$ is covered by some clique in the family.
The {\it edge clique cover number} $\theta_E(G)$ of a graph $G$
is the minimum size of an edge clique cover of $G$.
An edge clique cover of $G$
is called a {\it minimum edge clique cover} of $G$
if its size is equal to $\theta_E(G)$.
A \textit{vertex clique cover} of a graph $G$
is a family of cliques of $G$ such that each vertex of $G$
is contained in some clique in the family.
The smallest size of a vertex clique cover of $G$ is called
the {\it vertex clique cover number}, and is denoted by $\theta_V(G)$.

We denote a path with $n$ vertices by $P_n$,
a cycle with $n$ vertices by $C_n$, a graph with $n$ vertices
and no edges by $I_n$, and
a complete multipartite graph by $K_{n_1, \ldots, n_m}$.

\section{Proofs of Theorems \ref{thm:H(2,q)} and \ref{thm:H(3,q)}}
\subsection{Cliques in a Hamming graph}

Let $\pi_j:[q]^n \to [q]^{n-1}$ be a map defined by
\[
(x_1, ..., x_{j-1}, x_j, x_{j+1}, ..., x_n) \mapsto
(x_1, ..., x_{j-1}, x_{j+1}, ..., x_n).
\]
For $j \in [n]$ and $p \in [q]^{n-1}$,
let
\begin{equation}\label{eq:Sjp}
S_j(p) := \pi^{-1}_{j}(p) = \{x \in [q]^n \mid \pi_j(x)= p \}.
\end{equation}
Then $S_j(p)$ is a clique of $H(n,q)$ with size $q$.
Let
\begin{equation}\label{eq:Fnq}
\cF(n,q):=\{S_j(p) \mid j \in [n], p \in [q]^{n-1} \}.
\end{equation}
Then $\cF(n,q)$ is the family of maximal cliques of $H(n,q)$.

\begin{Lem}\label{lem:LB2}
Let $n\ge 2$ and $q \ge 2$,
and let $K$ be a clique of $H(n,q)$ with size at least $2$.
Then there exists a unique maximal clique $S$ of $H(n,q)$
containing $K$.
\end{Lem}

\begin{proof}
Since $\cF(n,q)$ is the family of maximal cliques of $H(n,q)$,
it is sufficient to show that
there is a unique maximal clique in $\cF(n,q)$ containing $K$.

Take a vertex $x = (x_1, x_2 \ldots, x_n) \in K$.
Now we will show that there exists a unique integer $j$
such that $\pi_j(x)=\pi_j(y)$ for all vertices $y\in K\setminus\{x\}$.
Take a vertex $y = ( y_1, y_2, \ldots, y_n ) \in K \setminus \{ x \}$.
Since $K$ is a clique,
$x$ and $y$ are adjacent.
Then there is a unique integer $j\in [n]$ such that
$x_j \neq y_j$ and $\pi_j(x)=\pi_j(y)$.
Suppose that there is a vertex
$z=(z_1, z_2, \ldots, z_n) \in K \setminus \{x,y\}$
such that $\pi_j(z)\not=\pi_j(x)$.
Since $x$ and $z$ are adjacent,
there is $j_1 \in [n]$ with $j_1 \neq j$
such that $\pi_{j_1}(x)=\pi_{j_1}(z)$,
and thus $x_j = z_j$.
Since $y$ and $z$ are adjacent,
there is $j_2 \in [n]$ with $j_2 \neq j$
such that
$\pi_{j_2}(y)=\pi_{j_2}(z)$,
and thus $y_j=z_j$.
Thus we have $x_j=z_j=y_j$,
which contradicts to the fact that $x_j\not=y_j$.
Therefore, $\pi_j(z)=\pi_j(x)$.

It implies that  $j$ is the unique integer
such that $\pi_j(x)=\pi_j(y)$ for all $y \in K$.
Hence $K$ is contained in $S_j(  \pi_j(x) ) \in \cF(n,q)$.
From the uniqueness of $j \in [n]$ and the fact that
$\pi_j(x)$ does not depend on the choice of $x \in K$,
it follows that $S_j(  \pi_j(x) )$ is
the unique maximal clique containing $K$.
\end{proof}

\begin{Lem}\label{lem:ecc}
The following hold: 
\begin{itemize}
\item[{\rm (a)}]
The family $\cF(n,q)$ defined by
(\ref{eq:Fnq})
is an edge clique cover of $H(n,q)$.
\item[{\rm (b)}]
The edge clique cover number of $H(n,q)$ is equal to $nq^{n-1}$.
\item[{\rm (c)}]
Any minimum edge clique cover of $H(n,q)$ consists of
edge disjoint maximum cliques.
\end{itemize}
\end{Lem}

\begin{proof}
Since each edge is contained in a maximal clique
and $\cF(n.q)$ is the family of maximal cliques in $H(n,q)$,
it follows that
$\cF(n,q)$ is an edge clique cover of $H(n,q)$.

Let $\cE$ be a minimum edge clique cover of $H(n,q)$,
that is, $\theta_E(H(n,q))=|\cE|$.
Since $\cF(n,q)$ is an edge clique cover with $|\cF(n,q)| =nq^{n-1}$,
we have $|\cE| \leq  nq^{n-1}$.
Now we will show that $|\cE| \geq nq^{n-1}$.
For a clique $S$, let $E(S):={S \choose 2}$.
Since $\cE$ is an edge clique cover of $H(n,q)$, it holds that
\begin{equation}\label{eq:minecc1}
|E(H(n,q))| \ \leq \ \sum_{S \in \cE} |E(S)|,
\end{equation}
and the equality holds if and only if
none of two distinct cliques in $\cE$ have a common edge.
Since the maximum size of a clique of $H(n,q)$ is equal to $q$,
we have $|E(S)|\le {q \choose 2}$ for each ${S\in \cE}$.
Therefore,
\begin{equation}\label{eq:minecc2}
\sum_{S \in \cE} |E(S)| \ \leq \ {q \choose 2} \times |\cE|,
\end{equation}
and the equality holds if and only if
any element of $\cE$ is a maximum clique in $H(n,q)$.
Since $|E(H(n,q))|= \frac{1}{2} n(q-1)q^n
= {q \choose 2} \times n q^{n-1}$,
it follows from (\ref{eq:minecc1})
and (\ref{eq:minecc2}) that
$nq^{n-1}\leq |\cE|$, or $nq^{n-1}=|\cE|$.

Moreover, since two equalities of (\ref{eq:minecc1})
and (\ref{eq:minecc2}) hold,
we can conclude that
any minimum edge clique cover of $H(n,q)$ consists of
edge disjoint maximum cliques.
\end{proof}

\begin{Cor}
The family $\cF(n,q)$ defined by
(\ref{eq:Fnq})
is a minimum edge clique cover of $H(n,q)$.
\end{Cor}

\begin{proof}
It follows from the fact that $|\cF(n,q)| =n q^{n-1}$ and
Lemma \ref{lem:ecc}.
\end{proof}

\subsection{Proof of Theorem \ref{thm:H(2,q)}}

In this subsection,
we give
a proof of Theorem \ref{thm:H(2,q)}.

\begin{Lem}\label{lem:locvcc}
Let $n\ge 2$ and $q \ge 2$.
For any vertex $x$ of $H(n,q)$,
we have
$\theta_V(N_{H(n,q)}(x))=n$.
\end{Lem}

\begin{proof}
Take any vertex $x \in {[q]^n}$ of $H(n,q)$.
Then the vertex $x$ is adjacent to a vertex $y$ such that
$\pi_j(x) = \pi_j(y)$ for some $j\in [n]$.
We can easily check from the definition of $H(n,q)$ that,
for any $j \in [n]$,
the set $S_j(  \pi_j(x) ):=\{y \in [q]^n \mid \pi_j(x) = \pi_j(y) \}$
forms a clique of $H(n,q)$.
Since $N_{H(n,q)}(x)=\cup_{j \in [n]} S_j( \pi_j(x) ) \setminus \{x\}$,
the family $\{S_j(  \pi_j(x) ) \mid j \in [n]\}$ is a vertex clique cover of
$N_{H(n,q)}(x)$ and so $\theta_V(N_{H(n,q)}(x))\le n$.

Moreover, note that
$S_j(\pi_j(x)) \cap S_{j^{\prime} } (\pi_{j^{\prime}}(x)) =\{x\}$
for $j,j^{\prime}\in [n]$ where $j\not=j^{\prime}$.
Take $y_j\in S_j(\pi_j(x))\setminus\{x\}$ for each $j\in[n]$.
Then $y_1,y_2,\ldots,y_n$ are $n$ vertices of $N_{H(n,q)}(x)$
such that no two of them can be covered by a same clique
and so $\theta_V(N_{H(n,q)}(x))\ge n$.
\end{proof}

Opsut showed the following lower bound for the competition number of a graph.

\begin{Thm}[\cite{MR679638}]\label{thm:OpsutBd1}
For a graph $G$, it holds that
$k(G) \geq \text{{\rm min}} \{ \theta_V(N_{G}(v)) \mid v \in V(G) \}$.
\end{Thm}

\begin{Cor}\label{cor:LB}
If $n \geq 2$ and $q \geq 2$,
then $k(H(n,q)) \geq n$.
\end{Cor}

\begin{proof}
It immediately follows from
Lemma \ref{lem:locvcc} and Theorem \ref{thm:OpsutBd1}.
\end{proof}

We define a total order $\prec$ on the set $[q]^n$
as follows.
Take two distinct elements
$x=(x_1, x_2, \ldots, x_n)$ and
$y=(y_1, y_2, \ldots, y_n )$
in $[q]^n$.
Then we define $x \prec y$
if there exists $j \in [n]$
such that $x_i=y_i$ for $i \leq j-1$
and $x_j < y_j$.
The \textit{lexicographic ordering} of $[q]^n$
is the ordering $v_{1}, v_{2}, \ldots, v_{q^n}$
such that $v_1 \prec v_2 \prec \ldots \prec v_{q^n}$.


\begin{proof}[Proof of Theorem \ref{thm:H(2,q)}]
By Corollary \ref{cor:LB}, it follows that $k(H(2,q)) \ge 2$.
Now we show that  $k(H(2,q)) \le 2$.
We define a digraph $D$ as follows: 
\begin{eqnarray*}
V(D) &=& V(H(2,q)) \cup \{ z_1, z_2 \}, \\
A(D) &=&
\left(
\bigcup_{i=2}^{q} \{ (x, (1,i-1) ) \mid x \in S_1(i) \}
\right)
\cup \left(
\bigcup_{i=2}^{q} \{ (x, (i-1,q) ) \mid x \in S_2(i) \}
\right) \\
&& \cup \{ (x, z_{1}) \mid x \in S_1(1) \}
\cup \{ (x, z_{2}) \mid x \in S_2(1) \},
\end{eqnarray*}
where $S_j(i)$ with $j \in \{1,2\}$ and $i \in [q]$ is
the clique of $H(2,q)$ defined by (\ref{eq:Sjp}).
Since $\{N^-_D(v) \mid v\in V(D), |N^-_D(v)|\ge 2 \} =\cF(2,q)$,
it is easy to check that $C(D)=H(2,q) \cup \{z_1,z_2\}$.
In addition, the ordering obtained by adding $z_1, z_2$
on the head of the 
lexicographic ordering of $V(H(2,q))$
is an acyclic ordering of $D$.
To see why, take an arc $(x, y)\in A(D)$.
If $x \in S_1(1)$ or $x \in S_2(1)$
then $y$ is either $z_1$ or $z_2$.
If $x \in S_1(i)$ or $x \in S_2(i)$ for some $2 \leq i \in[q]$,
then $x=(l,i)$ or $x=(i,l)$ for some $l \in [q]$.
Since $y=(1, i-1)$ or $y=(i-1, q)$, we have $y \prec x$.
Therefore $D$ is acyclic.
Hence we have $k(H(2,q)) \leq 2$.
\end{proof}


\subsection{Proof of Theorem \ref{thm:H(3,q)}}

In this subsection,
we give a proof of Theorem \ref{thm:H(3,q)}.

\subsubsection{}{$ k(H(3,q)) \geq 6$ $(q \geq 3)$}

First, we improve the lower bound for the competition number of
a Hamming graph $H(n,q)$
given in Corollary \ref{cor:LB}
in the case where $n \geq 3$ and $q \geq 3$.

\begin{Lem}\label{lem:LH_3}
Let $n\ge 2$ and $q \ge 2$.
Let $D$ be an acyclic digraph
such that $C(D)=H(n,q) \cup I_k$
with $I_k=\{z_1, z_2, \ldots, z_k\}$.
Let $z_1, z_2, \ldots, z_k, v_1, v_2, \ldots, v_{q^n}$
be an acyclic ordering of $D$.
Let $U_i :=\{v_1, \ldots, v_i \}$ for $i \in \{1, \ldots, q^n \}$.
Then
\[
|\{S \in \cF(n,q) \mid S \cap U_i \neq \emptyset \}| \leq  k+i-1.
\]
\end{Lem}

\begin{proof}
Let
\begin{eqnarray*}
\cN(D) &:=& \{N_{D}^-(x) \mid x \in V(D), |N_{D}^-(x)| \ge 2 \}, \\
\cS_i &:=& \{ N_{D}^-(x) \mid x \in U_{i-1}\cup I_k, |N_{D}^-(x)| \ge 2 \}, \\
\cK_i &:=& \{K \in \cN(D) \mid K \cap U_{i} \neq \emptyset \}.
\end{eqnarray*}
Since $D$ is acyclic, it holds that
$\cK_i = \{ K \in \cS_i \mid K \cap U_i \neq \emptyset \}$.
Since $|\cS_i| \leq k+i-1$, it follows that
\begin{eqnarray}\label{ueeq1}
|\cK_i| = |\{K \in \cS_i \mid K \cap U_i \neq \emptyset \}|
\leq |\cS_i| \leq k+i-1.
\end{eqnarray}
For each $K \in \cK_i$, there exists a unique element
in $\mathcal{F}(n,q)$ containing $K$ by Lemma \ref{lem:LB2},
we denote it by $S_K$.
From (\ref{ueeq1}), it remains to show that
\begin{eqnarray}\label{ueeq}
|\{ S \in \cF(n,q) \mid S \cap U_{i} \neq \emptyset \}| \leq |\cK_i|.
\end{eqnarray}
Take $S \in \cF(n,q)$ such that $S \cap U_{i} \neq \emptyset$.
Then there exists a vertex $x$ in $S \cap U_i$.
Since $q \geq 2$, there exists a vertex
$y \in S \setminus \{ x \}$.
Since $C(D)=H(n,q) \cup I_k$ and the vertices $x$ and $y$ are adjacent,
there is a common prey $u$ of $x$ and $y$ in $D$.
Then $x\in N_D^-(u)\cap U_i$ and so  $N_D^-(u) \in \cK_i$.
Since $N_D^-(u)$ contains $x$ and $y$,  $S_{N_D^-(u)}$
is a maximal clique containing $x$ and $y$.
Then both $S$ and $S_{N_D^-(u)}$ are maximal cliques containing $x$ and $y$.
By Lemma \ref{lem:LB2}, we have $S=S_{N_D^-(u)}$,
which implies that
$S \in \{ S_K \mid K \in \cK_i \}$.
It follows that
\[
\{S \in \cF(n,q) \mid S \cap U_{i} \neq \emptyset \}
\subseteq \{ S_K \mid K \in \cK_i \},
\]
and together with
$| \{ S_K \mid K \in \cK_i \}| \le |\cK_i| $,
(\ref{ueeq}) holds.
Hence, the lemma holds.
\end{proof}

\begin{Lem}\label{Lem:LB33}
For $n \geq 3$ and $q \geq 3$, we have $k(H(n,q)) \geq 3n-4$.
\end{Lem}

\begin{proof}
Let $k$ be the competition number of $H(n,q)$ and
let $D$ be an acyclic digraph such that $C(D)=H(n,q) \cup I_k$
with $I_k=\{z_1, z_2, \ldots, z_k\}$.
Let $z_1$, $z_2$, $\ldots$, $z_k$, $v_1$, $v_2$, $\ldots$, $v_{q^n}$
be an acyclic ordering of $D$.
Let $ U_3 :=\{v_1, v_2, v_3 \}$.
By Lemma \ref{lem:LH_3}, it holds that
\begin{equation}
|\{S \in \cF(n,q) \mid S \cap U_3 \neq \emptyset \}| \leq k+2.
\end{equation}
In addition, it holds that
$|\{S \in \cF(n,q) \mid S \cap U_3 \neq \emptyset \}|\ge 3n-2$
whose proof will be shown in next paragraph.
Therefore, we have $3n-2 \leq k+2$, or $k \geq 3n-4$.

Now it remains to show that
$|\{S \in \cF(n,q) \mid S \cap U_3 \neq \emptyset \}|\ge 3n-2$.
Consider the subgraph of $H(n,q)$ induced by $U_3$, say $H$.
Then $H$ is isomorphic to one of the following:
\[
\text{(i)} \ K_3 \qquad
\text{(ii)} \ P_3  \qquad
\text{(iii)} \ P_2 \cup I_1 \qquad
\text{(iv)} \ I_3.
\]

\noindent
{\bf Case (i) $H \cong K_3$}:
By Lemma \ref{lem:LB2}, $U_3$ is contained in exactly one maximal clique.
Without loss of generality,
we may assume that $U_3$ is contained in
$S_1((\underbrace{1,\ldots,1}_{n-1}) )$,
and so we may also assume that
\[
U_3=\{ (\underbrace{1,1,\ldots,1}_{n}),
(2,\underbrace{1,\ldots,1}_{n-1}),
(3,\underbrace{1,\ldots,1}_{n-1}) \}.
\]
Then the family $\{S \in \cF(n,q) \mid S \cap U_3 \neq \emptyset \}$
consists of the following $3n-2$ elements:
\begin{eqnarray*}
&& S_1((\underbrace{1,\ldots,1}_{n-1})), \quad
S_j((i,\underbrace{1,\ldots,1}_{n-2})) \
(i \in \{1,2,3\}, \ j \in [n] \setminus \{1\}).
\end{eqnarray*}

\noindent
{\bf Case (ii) $H\cong P_3$}:
Without loss of generality, we may assume that
\[
U_3 =\{ (\underbrace{1,\ldots,1}_{n}),
(2,\underbrace{1,\ldots,1}_{n-1}),
(1,2,\underbrace{1,\ldots,1}_{n-2}) \}.
\]
Then the family $\{S \in \cF(n,q) \mid S \cap U_3 \neq \emptyset \}$
consists of the following $3n-2$ elements:
\begin{eqnarray*}
&& S_j((\underbrace{1,\ldots,1}_{n-1})) \
(j \in [n]), \quad
S_j( (2,\underbrace{1,\ldots,1}_{n-2})) \
(j \in [n] \setminus \{1\}), \\
&& S_j( (1,2,\underbrace{1,\ldots,1}_{n-3} )) \
(j \in [n] \setminus \{2\}).
\end{eqnarray*}

\noindent
{\bf Case (iii) $H\cong P_2\cup I_1$ or (iv) $H\cong I_3$}:
Let $v$ be an isolated vertex of $H$.
Since the $n$ cliques in $\cF(n,q)$ containing $v$
do not contain
the other vertices of $U_3$,
it is sufficient to show that
$\{S \in \cF(n,q) \mid S \cap (U_3\setminus\{v\}) \neq \emptyset \}$
has at least $2n-2$ elements.
Since,
for each vertex $u \in U_3\setminus\{v\}$,
there are $n$ cliques in $\cF(n,q)$ containing $u$ and
there is at most one clique in $\cF(n,q)$ containing the two vertices
of $U_3\setminus\{v\}$.
Thus we can conclude that
$\{S \in \cF(n,q) \mid S \cap (U_3 \setminus \{v\}) \neq \emptyset \}$
has at least $2n-1$ elements.

We complete the proof.
\end{proof}

If $n=3$, then the above lower bound gives $k(H(3,q)) \geq 5$.
In this case, however, we can improve the bound as follows.

\begin{Lem}\label{lem:LBH_3q}
For $q \geq 3$, we have $k(H(3,q)) \ge 6$.
\end{Lem}

\begin{proof}
By Lemma \ref{Lem:LB33}, we have $k(H(3,q)) \ge 5$.
Suppose that $k(H(3,q)) = 5$.
Then there exists an acyclic digraph $D$
such that $C(D)=H(3,q) \cup I_5$
with $I_5=\{z_1, z_2, \ldots, z_5\}$.
Let $z_1, z_2, \ldots, z_5, v_1, v_2, \ldots, v_{q^3}$
be an acyclic ordering of $D$.
Let $U_4 :=\{v_1, v_2, v_3, v_4 \}$.
For convenience, let
\[
\cA_1 := \{S \in \cF(3,q) \mid S \cap U_4 \neq \emptyset \}, \quad
\cA_2 := \{S \in \cF(3,q) \mid v_5 \in S \}.
\]

Now we consider
the subgraph $G$ of $H(3,q)$ induced by $U_4$.
Any graph on $4$ vertices is isomorphic to one of
the following graphs: 
\[
\begin{array}{llll}
\text{(i)} \ K_4 &
\text{(ii)} \ K_{1,1,2} &
\text{(iii)} \ K_4 - E(P_3) &
\text{(iv)} \ C_4 \\
\text{(v)} \ P_4 &
\text{(vi)} \ K_{1,3} &
\text{(vii)} \ K_3 \cup I_1 &
\text{(viii)} \ K_2 \cup K_2 \\
\text{(ix)} \ P_3 \cup I_1 &
\text{(x)} \ K_2 \cup I_2 &
\text{(xi)} \ I_4 &
\end{array}
\]
Since $H(3,q)$ does not contain
an induced subgraph isomorphic to $K_{1,1,2}$ by Lemma \ref{lem:LB2},
$G$ is one of the above graphs except (ii).
For each cases,
the number $|\cA_1|$
is given as follows: 
\[
\begin{array}{llll}
\text{(i)}  \ 9 &
\text{(ii)}  \ - &
\text{(iii)}  \ 9 &
\text{(iv)}  \ 8 \\
\text{(v)}  \ 9 &  \text{(vi)}  \ 9 &
\text{(vii)}  \ 10 &
\text{(viii)}  \ 10 \\
\text{(ix)}  \ 10 &
\text{(x)}  \ 11 &
\text{(xi)}  \ 12 &
\end{array} \]
By Lemma \ref{lem:LH_3}, we have
$|\cA_1 | \leq 8$.
Therefore $G \cong C_4$ and so $|\cA_1| = 8$.
Since each vertex of $H(3,q)$ is contained in exactly
$3$ cliques in $\cF(3,q)$, $|\cA_2|=3$.
From the fact that
\[
\cA_1 \cup \cA_2=\{S \in \cF(3,q) \mid S \cap
(U_4 \cup \{v_5\}) \neq \emptyset \},
\]
it holds that $|\cA_1 \cup \cA_2| \leq 9$ by Lemma \ref{lem:LH_3}.
Since $|\cA_1|=8$, $|\cA_2|=3$, and $|\cA_1 \cup \cA_2| \leq 9$,
we have
$|\cA_1\cap \cA_2|=|\cA_1|+|\cA_2|-|\cA_1\cup \cA_2|
\geq 8+3-9 = 2$.

Take two distinct cliques
$S, S^{\prime} \in \cA_1 \cap \cA_2.$
Then $S \cap U_4 \neq \emptyset$,
$S^{\prime}\cap U_4\not=\emptyset$
and so take $x\in S\cap U_4$ and
$y\in S^{\prime}\cap U_4$.
If $x=y$ or $x$ and $y$ are adjacent,
then $S=S^{\prime}$ by Lemma \ref{lem:LB2}.
Therefore $x$ and $y$ are not adjacent.
Since $G \cong C_4$,
without loss of generality,
we may assume that
\[
U_4 = \{ (1,1,1),(1,1,2), (1,2,2), (1,2,1) \}, \quad
x = (1,1,1), \quad y = (1,2,2).
\]
Since $x$ and $v_5$ are adjacent,
one of the following holds:
\begin{eqnarray}
\begin{split}
& \pi_1(x)=(1,1)=\pi_1(v_5), \quad
\pi_2(x)=(1,1)=\pi_2(v_5), \\
& \pi_3(x)=(1,1)=\pi_3(v_5).
\end{split} \label{eqL6} 
\end{eqnarray}
Since $y$ and $v_5$ are adjacent,
one of the following holds:
\begin{eqnarray}
\begin{split}
& \pi_1(y)=(2,2)=\pi_1(v_5), \quad
\pi_2(y)=(1,2)=\pi_2(v_5), \\
&\pi_3(y)=(1,2)=\pi_3(v_5). 
\end{split} \label{eqL7} 
\end{eqnarray}
However, it is impossible that
$v_5$ satisfies both one of (\ref{eqL6}) and one of (\ref{eqL7})
since $v_5 \not\in U_4$.
We reach a contradiction.
Hence we conclude $k(H(3,q)) \geq 6$.
\end{proof}

\subsubsection{}{$ k(H(3,q)) \leq 6$ $(q \geq 3)$}

Next, we show the upper bound $k(H(3,q)) \leq 6$ for $q \geq 3$.
To show the upper bound,
we introduce a graph
$K_{q_1} \Box K_{q_2} \Box K_{q_3}$
as an extension of $H(3,q)$.

For graphs $G$ and $H$,
the {\it Cartesian product} $G \Box H$ of $G$ and $H$ is
the graph which has the vertex set $V(G)\times V(H)$
and has an edge between two vertices $(g,h)$ and $(g',h')$
if and only if
$gg'\in E(G)$ and $h=h'$, or $g=g'$ and $hh'\in E(H)$.
Note that the Cartesian product of $n$ complete graphs $K_q$ of size $q$
is
the Hamming graph $H(n,q)$.

Let $q_1, q_2, q_3 \geq 2$ be integers and we consider the graph
$K_{q_1} \Box K_{q_2} \Box K_{q_3}$.
Define
\begin{eqnarray*}
\pi_1 : [q_1] \times [q_2]\times [q_3] \to [q_2] \times [q_3],
&\ & (x_1, x_2, x_3) \mapsto (x_2, x_3), \\
\pi_2 : [q_1] \times [q_2]\times [q_3] \to [q_1] \times [q_3],
&\ & (x_1, x_2, x_3) \mapsto (x_1, x_3), \\
\pi_3 : [q_1] \times [q_2]\times [q_3] \to [q_1] \times [q_2],
&\ & (x_1, x_2, x_3) \mapsto (x_1, x_2).
\end{eqnarray*}
For
$p_1 \in [q_1]$, $p_2 \in [q_2]$, and $p_3 \in [q_3]$, let
\begin{eqnarray*}
S_1((p_2,p_3)) &:=& \{ x \in [q_1] \times [q_2] \times [q_3]
\mid \pi_1(x) = (p_2,p_3) \}, \\
S_2((p_1,p_3)) &:=& \{ x \in [q_1] \times [q_2] \times [q_3]
\mid \pi_2(x) = (p_1,p_3) \}, \\
S_3((p_1,p_2)) &:=& \{ x \in [q_1] \times [q_2] \times [q_3]
\mid \pi_3(x) = (p_1,p_2) \}.
\end{eqnarray*}
Note that $S_1((p_2,p_3))$, $S_2((p_1,p_3))$, and $S_3((p_1,p_2))$
are maximal cliques of $K_{q_1} \Box K_{q_2} \Box K_{q_3}$.
We denote
the set of all maximal cliques
$S_1((p_2,p_3))$, $S_2((p_1,p_3))$ and $S_3((p_1,p_2))$
by $\cF_{(q_1,q_2,q_3)}$.
Then $\cF_{(q_1,q_2,q_3)}$ is an edge clique cover of
$K_{q_1} \Box K_{q_2} \Box K_{q_3}$.

\begin{Lem} \label{lem:UH_3q}
For $q_1,q_2,q_3 \geq 2$,
there exists an acyclic digraph $D$ such that
$C(D) = (K_{q_1} \Box K_{q_2} \Box K_{q_3}) \cup I_6$
and
\[
\{ N_D^-(v) \mid v \in V(D), |N_D^-(v)| \geq 2 \}
= \cF_{(q_1,q_2,q_3)}.
\]
Consequently, we have
$k(K_{q_1} \Box K_{q_2} \Box K_{q_3}) \leq 6$.
\end{Lem}

\begin{proof}
For any digraph $D$, we define
$\cN(D):= \{ N_D^-(v) \mid v \in V(D),| N_D^-(v)| \ge 2 \}$.
We prove the lemma by induction on $m = q_1 + q_2 + q_3$.
Since ${q_1}, {q_2}, {q_3}  \ge 2$, we have $m\ge 6$.
Suppose $m=6$, i.e., $q_1=q_2=q_3=2$.
Note that $K_{2} \Box K_{2} \Box K_{2} = H(3,2)$.
Since $H(3,2)$ is a triangle-free graph, there exists an
acyclic digraph $D$ such that $C(D)=H(3,2) \cup I_6$ and that
$N_D^-(v)$ is either the empty set or a maximum clique in $H(3,2)$
for each vertex $v \in V(D)$ (see Figure \ref{fig1} for an
illustration of such a digraph).
Thus the statement is true for $m=6$.

\begin{figure}
\begin{center}
\psfrag{A}{\small(1,2,2)}
\psfrag{B}{\small(2,2,2)}
\psfrag{C}{\small(1,1,2)}
\psfrag{D}{\small(2,1,2)}
\psfrag{E}{\small(1,2,1)}
\psfrag{F}{\small(2,2,1)}
\psfrag{G}{\small(1,1,1)}
\psfrag{H}{\small(2,1,1)}
\psfrag{I}{$H(3,2)$}
\psfrag{J}{$D$}
\includegraphics[width=310pt]{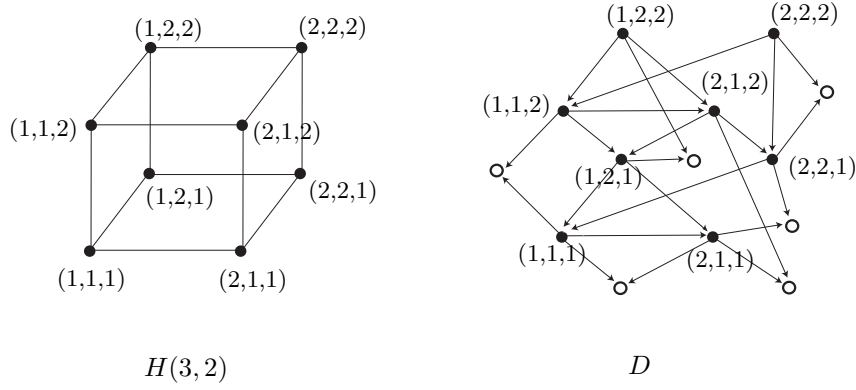}\\
  \caption{The Hamming graph $H(3,2)$ and an acyclic digraph $D$ satisfying
$C(D)=H(3,2) \cup I_6$}
\label{fig1}
\end{center}
\end{figure}

Suppose that the statement is true for $m= q_1 + q_2 + q_3$
where $m \geq 6$.
Consider a graph $K_{q_1} \Box K_{q_2} \Box K_{q_3}$
such that $m+1 = q_1+q_2+q_3$.
Since $q_1+q_2+q_3 > 6$,
at least one of $q_1$, $q_2$, or $q_3$ is
greater than $2$.
Without loss of generality, we may assume that $q_1>2$.
Now we consider the graph $K_{q_1-1} \Box K_{q_2} \Box K_{q_3}$
which is a subgraph of $K_{q_1} \Box K_{q_2} \Box K_{q_3}$.
Then by the induction hypothesis,
there exists an acyclic digraph $D_0$ such that
$C(D_0)=(K_{q_1-1} \Box K_{q_2} \Box K_{q_3}) \cup I_6$ and
\begin{equation}\label{eeqq}
\cN(D_0)=\cF_{(q_1 -1,q_2,q_3)}.
\end{equation}
Let $v_1,v_2,\ldots, v_{q_1q_2q_3 - q_2q_3 + 6}$
be an acyclic ordering of $D_0$.
For convenience,
let $w_1:=v_{q_1q_2q_3-q_2q_3+5}$ and $w_2:=v_{q_1q_2q_3-q_2q_3+6}$.

Let $H^*$ be the subgraph of $K_{q_1} \Box K_{q_2} \Box K_{q_3}$ induced by
\begin{eqnarray*}
V^*&:=& V(K_{q_1} \Box K_{q_2} \Box K_{q_3})
- V(K_{q_1-1} \Box K_{q_2} \Box K_{q_3}) \\
&=& \{(q_1, p_2, p_3) \mid p_2 \in [q_2], p_3 \in [q_3] \}.
\end{eqnarray*}
Now we define a digraph $D_1$ as follows: 
\begin{eqnarray*}
V(D_1) &=& V^* \cup \{ w_1, w_2 \}, \\
A(D_1) &=& \left( \bigcup_{i=2}^{q_3}
\{ (x, (q_1,1,i-1) ) \mid x \in S_2((q_1,i)) \}  \right) \\
&& \cup \left( \bigcup_{i=2}^{q_2}
\{ (x, (q_1,i-1,q) ) \mid x \in S_3((q_1,i)) \} \right) \\
&& \cup \ \{ (x, w_1 ) \mid x \in S_2((q_1,1)) \}
\ \cup \ \{ (x, w_2) \mid x \in S_3((q_1,1)) \}.
\end{eqnarray*}
The ordering obtained by adding $w_1,w_2$ on the head of the
lexicographic ordering of $V^*$ is an acyclic ordering of $D_1$,
and let $ w_1, w_2, \ldots, w_{q_2q_3+2}$ be the ordering.
In addition,
\begin{equation}\label{eeqq2}
\cN(D_1) =
\{ S_2((q_1,i)) \mid i \in [q_3] \}
\cup
\{ S_3((q_1,i')) \mid i' \in [q_2] \}.
\end{equation}
Therefore $D_1$ is an acyclic digraph such that
$C(D_1)=H^* \cup \{ w_1, w_2 \}$.

Note that, for $(p_2,p_3) \in [q_2] \times [q_3]$,
the clique in $K_{q_1} \Box K_{q_2} \Box K_{q_3}$ obtained by deleting
the vertex $(q_1,p_2,p_3)$ from an element
$S_1((p_2,p_3))$ of $K_{q_1} \Box K_{q_2} \Box K_{q_3}$ is
a maximal clique of $K_{q_1-1} \Box K_{q_2} \Box K_{q_3}$.
Then by (\ref{eeqq}), for each $(p_2,p_3) \in [q_2] \times [q_3]$,
the set
$\{ v \in V(D_0) \mid N_{D_0}^-(v) =
S_1((p_2,p_3)) \setminus \{ (q_1,p_2,p_3) \} \}$ is not empty,
and so we take an element $y_{(p_2,p_3)}$ of this set.

Now we define a digraph $D$ as follows: 
\begin{eqnarray*}
V(D)&=&V(K_{q_1} \Box K_{q_2} \Box K_{q_3})\cup I_6 \\
A(D)&=& A(D_0) \cup A(D_1)
\cup \left\{ ( (q_1,p_2,p_3), y_{(p_2,p_3)} )
\mid (p_2,p_3) \in [q_2] \times [q_3]
\right\}
\end{eqnarray*}
Note that since $N_{D_0}^-(w_1)=N_{D_0}^-(w_2)=\emptyset$,
\begin{equation}\label{eeqq1}
\{ N_{D}^-(w_1), N_{D}^-(w_2) \} = \{ S_2((q_1,1)), S_3((q_1,1)) \}.
\end{equation}
By the definition of $D$
and (\ref{eeqq}), (\ref{eeqq2}), and (\ref{eeqq1}),
we can conclude that
$\cN(D) = \cF_{(q_1,q_2,q_3)}$ and so $E(C(D))
= E(K_{q_1} \Box K_{q_2} \Box K_{q_3})$.
Thus, $C(D) = (K_{q_1} \Box K_{q_2} \Box K_{q_3}) \cup I_6$.
Then the ordering
\begin{eqnarray}\label{order}
v_1, v_2, \ldots, v_{q_1q_2q_3-q_2q_3+6}, w_3, w_4, \ldots, w_{q_2q_3+2}
\end{eqnarray}
of the vertices of $D$ 
is an acyclic ordering. 
To see this, take an arc $a=(x,y)\in A(D)$.
If $a \in A(D_0)\cup A(D_1)$, 
then $y$ 
appears before 
$x$
in (\ref{order}), 
since $D_0$ and $D_1$ are acyclic.
If $a\not \in A(D_0)\cup A(D_1)$
then $x \in \{w_3,w_4,\ldots,w_{q_2q_3+2} \}$ and
$y\in \{v_1,v_2,\ldots,v_{q_1q_2q_3-q_2q_3+6}\}$, 
thus $y$ 
appears before 
$x$ in (\ref{order}).
Thus the digraph $D$ is acyclic.
Hence the lemma holds.
\end{proof}

\begin{proof}[Proof of Theorem \ref{thm:H(3,q)}]
By Lemma \ref{lem:LBH_3q}, we have $k(H(3,q)) \geq 6$ for $q \geq 3$.
By Lemma \ref{lem:UH_3q}, we have $k(H(3,q))
=k(K_{q} \Box K_{q} \Box K_{q}) \leq 6$ for $q \geq 3$.
Hence Theorem \ref{thm:H(3,q)} holds.
\end{proof}

\section{Concluding Remarks}

In this paper, we gave the exact values of
the competition numbers of Hamming graphs
with diameter $2$ or $3$.

We conclude this paper with leaving the following
questions for further study: 
\begin{itemize}
\item{}
What is the competition number of a
Hamming graph $H(4,q)$ with diameter $4$ for $q \geq 3$?

\item{}
What is the competition number of a ternary
Hamming graph $H(n,3)$ for $n \geq 4$?

\item{}
Give the exact values or a good bound for the competition numbers of
Hamming graphs $H(n,q)$.
\end{itemize}


\end{document}